\documentclass[12pt]{amsart}
\usepackage[utf8]{inputenc}
\usepackage{graphicx}
\usepackage{epstopdf}
\usepackage{inputenc}
\usepackage{amsmath}
\usepackage{xcolor}
\usepackage{amssymb}
\newtheorem{theorem}{Theorem}
\newtheorem*{theorem*}{Theorem}

\newtheorem{corollary}{Corollary}[section]
\newtheorem{lemma}{Lemma}

\newtheorem*{acknowledgements*}{Acknowledgements}
\newtheorem{remark}{Remark}

\makeatletter
\def\blfootnote{\gdef\@thefnmark{}\@footnotetext}
\makeatother

\def\house#1{\setbox1=\hbox{$\,#1\,$}%
\dimen1=\ht1 \advance\dimen1 by 2pt \dimen2=\dp1 \advance\dimen2 by 2pt
\setbox1=\hbox{\vrule height\dimen1 depth\dimen2\box1\vrule}%
\setbox1=\vbox{\hrule\box1}%
\advance\dimen1 by .4pt \ht1=\dimen1
\advance\dimen2 by .4pt \dp1=\dimen2 \box1\relax}

\begin{document}
\title{Mean values and moments of arithmetic functions over number fields}
\author{Jaitra Chattopadhyay and Pranendu Darbar}
   \address[Jaitra Chattopadhyay]{Harish-Chandra Research Institute, HBNI\\
Chhatnag Road, Jhunsi\\
Allahabad 211019, India}
\address[Pranendu Darbar]{Institute of Mathematical Sciences, HBNI
\\CIT Campus, Taramani\\
Chennai-600113, India}
\email[Jaitra Chattopadhyay]{jaitrachattopadhyay@hri.res.in}
\email[Pranendu Darbar]{darbarpranendu100@gmail.com}

\begin{abstract}
For an odd integer $d > 1$ and a finite Galois extension $K/\mathbb{Q}$ of degree $d$, G. L\"{u} and Z. Yang \cite{lu3} obtained an asymptotic formula for the mean values of the divisor function for $K$ over square integers. In this article, we obtain the same for finitely many number fields of odd degree and pairwise coprime discriminants, together with the moment of the error term arising here, following the method adapted by S. Shi in \cite{shi}. We also define the sum of divisor function over number fields and find the asymptotic behaviour of the summatory function of two number fields taken together.

\end{abstract}
\maketitle

\section{Introduction}

\noindent
By an arithmetic function $f,$ we mean a function $f$ from $\mathbb{N}$ to $\mathbb{C}$. These functions appear in many branches of mathematics, particularly in analytic number theory and thus it is extremely useful to know the behaviour of them. One of the many arithmetic functions of special interest to number theorists is the {\it divisor function}, that counts the number of divisors of a natural number $n$, and is defined by the following equation.
\begin{equation*}
\tau (n)=\displaystyle\sum_{d \mid n} 1.
\end{equation*}

\vspace{2mm}
\noindent
By looking at a few small values of $\tau (n)$, we observe that the behaviour is quite irregular. But instead of $\tau (n)$, if we consider the quantity $\displaystyle\sum_{n \leq x}\tau (n)$, for a large positive real number $x$, we get a much well-behaved function of $x$. In fact, we have the following asymptotic formula
\begin{equation*}
\displaystyle\sum_{n \leq x} \tau (n) = x\log x + (2\gamma - 1)x + O(x^{\theta}),
\end{equation*}
where $\gamma$ is the Euler's constant and $\theta$ is a real number with $0<\theta < 1$. The determination of the exact order of the error term is famously known as the {\it Dirichlet divisor problem}.

\vspace{2mm}

\noindent
Likewise, we may consider analogous problems in algebraic number fields. For a number field $K$ and positive integers $k \mbox{ and } n$, it is a natural question to ask for the number of ways we can write $n$ as a product of the norms of $k$ ideals in the ring of integers $\mathcal{O}_K$ of $K$. More precisely, we define 
\begin{equation}\label{defn}
\tau_{k}^K(n) = \displaystyle\sum_{N(\mathfrak{a}_1\ldots \mathfrak{a}_k)=n}1,
\end{equation}
and study its asymptotic behaviour.
  
\vspace{2mm}

\noindent
It is easy to see from \eqref{defn} that $\tau_{k}^{K}(n)$ is a multiplicative function of $n$. The problem of finding an asymptotic formula for the function $\displaystyle\sum_{n \leq x}\tau_{k}^K(n)$ is known as the {\it k-dimensional divisor problem} in $K$. E. Panteleeva \cite{pant1} considered this problem for quadratic and cyclotomic fields, providing an asymptotic formula for both the cases as follows.

\begin{theorem} $(${\bf Quadratic field case}$)$ \cite{pant1}
Let $K=\mathbb{Q}(\sqrt{D})$ for some square-free integer $D$. Then for every integer $k \geq 1$, $$\displaystyle\sum_{n \leq x}\tau_{k}^K(n)=xP_{k}(\log x)+\theta x^{1-\frac{10}{133}k^{-\frac{2}{3}}}(C\log x)^{2k}$$ is valid for all real number $x$ satisfying $(\log x)^2 \geq D$, where $P_k$ is a polynomial of degree $k-1$, $\theta$ is a complex number satisfying $|\theta|\leq 1$ and $C$ is an absolute constant.
\end{theorem}

\begin{theorem} $(${\bf Cyclotomic field case}$)$ \cite{pant1}
Let $k$ and $t \geq 1$ be integers and let $\zeta_{t}$ be a primitive $t$-th root of unity. Then for $K=\mathbb{Q}(\zeta_{t})$, we have $$\displaystyle\sum_{n \leq x}\tau_{k}^K(n)=xP_{k}(\log x)+\theta x^{1-\frac{1}{12}(\phi (t) k)^{-\frac{2}{3}}}(C\log x)^{\phi (t)k},$$ where $P_k$ is a polynomial of degree $k-1$, $\theta$ is a complex number satisfying $|\theta|\leq 1$, $C$ is an absolute constant and $\phi$ stands for the Euler's phi-function.
\end{theorem}

\vspace{2mm}

\noindent
It is also interesting to deal with different positive integers $k_1, \ldots , k_{l-1} \mbox{ and } k_{l} \geq 2$ rather than just a single $k$. In \cite{pant2}, E. Panteleeva addressed this problem for $K=\mathbb{Q}$ and proved the following.

\begin{theorem} \cite{pant2} \label{pant}
Let $l \geq 1$ and $k_1, \ldots , k_l \geq 2$ be integers. Then $$\displaystyle\sum_{n \leq x}\tau_{k_1}(n)\ldots \tau_{k_l}(n)=xP_{m}(\log x)+\theta x^{1-\frac{2}{31}m^{-\frac{2}{3}}}(C\log x)^{m}$$ is valid for all $x$ such that $\log x \geq m$, where $m=\displaystyle\prod_{i=1}^{l}k_{i}$, $P_m$ is a polynomial of degree $m-1$, $\theta$ is a complex number satisfying $|\theta|\leq 1$ and $C$ is an absolute constant.
\end{theorem}

\vspace{2mm}

\noindent
In \cite{deza}, E. Deza and L. Varukhina extended Theorem \ref{pant}  to quadratic and cyclotomic fields. Later, G. L\"{u} \cite{lu1} generalized their results to a finite Galois extension of $\mathbb{Q}$ as follows.

\begin{theorem} \cite{lu1} \label{l1}
Let $K/\mathbb{Q}$ be a Galois extension of degree $d$. Then for integers $l \geq 1$, $k_1, \ldots , k_l \geq 2$ and any $\epsilon > 0$, we have $$\displaystyle\sum_{n \leq x}\tau^K_{k_1}(n)\ldots \tau^K_{k_l}(n)=xP_{m}(\log x)+O(x^{1-\frac{3}{md+6}+\epsilon}),$$ where $m=k_1\ldots k_ld^{l-1}$ and $P_m$ is a polynomial of degree $m-1$.
\end{theorem}

\noindent
Later, G. L\"{u} and W. Ma \cite{lu-ma} extended Theorem \ref{l1} for several number fields and proved the following.

\begin{theorem} \cite{lu-ma}
Let $l \geq 1$ and $k_1, \ldots, k_l \geq 2$ be given integers. For each $i \in \{1, 2, \ldots , l\}$, let $K_i$ be a number field having discriminant $D_i$. Suppose that $\gcd(D_i, D_j)=1$ for $i \neq j$. Then for any $\epsilon > 0$, we have $$\displaystyle\sum_{n \leq x}\tau^{K_1}_{k_1}(n)\ldots \tau^{K_l}_{k_l}(n)=xP_{m}(\log x)+O(x^{1-\frac{3}{md_1\ldots d_l+6}+\epsilon}),$$ where $m=k_1\ldots k_ld^{l-1}$ and $P_m$ is a polynomial of degree $m-1$.
\end{theorem}

\vspace{2mm}

\noindent
From the formula $$\displaystyle\sum_{n \leq x}\tau (n)=\displaystyle\sum_{n \leq x}\displaystyle\sum_{d \mid n} 1 = \displaystyle\sum_{ab \leq x}1,$$ one sees that the above expression counts the number of lattice points lying in the first quadrant as well as below the hyperbola $XY=x$. Similarly, the quantity $$\displaystyle\sum_{n \leq x}\tau (n^2)=\displaystyle\sum_{n \leq x}\displaystyle\sum_{d \mid n^2} 1 = \displaystyle\sum_{n \leq x}\displaystyle\sum_{ab=n^2} 1$$ counts the number of lattice points lying in the first quadrant as well as below the hyperbola $XY=x$, where $x$ is a perfect square. This geometric interpretation motivates to study the asymptotic behaviour of divisor functions over square integers. Along this direction, G. L\"{u} and Z. Yang \cite{lu3} proved the following.

\begin{theorem} \cite{lu3} \label{l3}
Let $d \geq 3$ be an odd integer and let $K$ be a Galois extension of degree $d$ over $\mathbb{Q}$. Then for any integer $k \geq 2$ and any $\epsilon > 0$, we have $$\displaystyle\sum_{n \leq x}\tau^{K}_{k}(n^2)=xP_m(\log x)+O(x^{1-\frac{3}{md+6}+\epsilon}),$$ where $m=\frac{k^2d+k}{2}$ and $P_m$ is a polynomial of degree $m-1$.
\end{theorem}

\vspace{2mm}

\noindent
We now consider the problem related to the order of magnitude of a function. In general, to determine the precise order of magnitude of an arithmetic function is very hard. But it is very useful to find the following limiting behaviour.
 \begin{align*}
 \limsup_{x\rightarrow \infty}\frac{f(x)}{g(x)}=+\infty \quad \text{ and } \quad  \liminf_{x\rightarrow \infty}\frac{f(x)}{g(x)}=-\infty
 \end{align*}
 for some  function $f(x)$ and a positive function $g(x).$
 
 \noindent
 For example, if $r_2(n)$ stands for the number of representations of $n$ as a sum of two squares, then it is known that 

\begin{equation*}
R(x) = \displaystyle\sum_{n \leq x} r_2(n) = \pi x + E(x).
\end{equation*}

\vspace{2mm}
\noindent
Hardy \cite{hardy} and Ingham \cite{ingham} proved that $$\limsup\limits_{x\to \infty} \frac{E(x)}{x^{\frac{1}{4}}} = +\infty \qquad \mbox{ and } \qquad \liminf\limits_{x\to \infty} \frac{E(x)}{x^{\frac{1}{4}}} = -\infty.$$   

\vspace{2mm}
\noindent
In \cite{kc-rn-annals}, K. Chandrasekharan and R. Narasimhan extended the above formulae for a large class of arithmetical functions. We now briefly describe their result proved in \cite{kc-rn-annals}.

\vspace{2mm}
\noindent
Suppose we have the following functional equation 
\begin{equation} \label{omega-type-eq-2}
\Delta (s) \phi (s) = \Delta (\delta - s) \psi (\delta - s),
\end{equation}

\vspace{2mm}
\noindent
where $\delta$ is a real number, $\Delta (s) = \displaystyle\prod_{\nu=1}^{N}\Gamma(\alpha _{\nu}s + \beta _{\nu})$, $\alpha_{\nu} > 0$ satisfying $M=\displaystyle\sum_{\nu = 1}^{N}\alpha _ {\nu} \geq 1$ and $\beta _{\nu}$'s are complex numbers. We consider solutions of \eqref{omega-type-eq-2} in the form of Dirichlet series. 

\medskip
\noindent
Let $\phi (s) = \displaystyle\sum_{n=1}^{\infty} \frac{a_n}{{\lambda^s _n}} \mbox{ for } Re(s) > \alpha$  and $\psi (s) = \displaystyle\sum_{n=1}^{\infty} \frac{b_n}{{\mu^s_n}} \mbox{ for } Re(s) > \beta$, where $\alpha \mbox{ and } \beta$ are complex numbers. Also, let $A(x) = \displaystyle\sum_{\lambda_n \leq x}a_n$, $B(x) = \displaystyle\sum_{\mu _n \leq x}b_n$ and $Q(x) = \frac{1}{2\pi i}\int_{C}\phi (s)\frac{x^s}{s} ds,$ where $C$ is a suitably chosen contour containing all the singularities of the integrand. The following is a special case of a theorem proved in \cite{kc-rn-annals}.

\begin{theorem} \cite{kc-rn-annals}\label{omega-type-theorem-kc-rn}
Let $\Delta, \phi , \psi, M$ be defined as above such that \eqref{omega-type-eq-2} holds. Let $\{\mu _n \}$ contain a subsequence $\{\mu _{n_k}\}$ such that $\mu _n ^{\frac{1}{2M}}$ is representable as a linear combination of the numbers $\{\mu _{n_k}^{\frac{1}{2M}}\}$ with coefficients $\pm 1$, unless $\mu _n ^{\frac{1}{2M}}=\mu _{n_r} ^{\frac{1}{2M}}$ for some $r$, in which case $\mu _n ^{\frac{1}{2M}}$ has no other representations. Suppose 

\begin{equation}\label{eq-3}
\displaystyle\sum_{n=1}^{\infty} \frac{|Re(b_{n_k})|}{\mu _{n_k}^{\frac{M \delta + \frac{1}{2}}{2M}}}  = +\infty \mbox { holds }.
\end{equation}
 
\noindent
Then we have 

\begin{equation}\label{desired-1}
\limsup\limits_{x\to \infty} \frac{Re(A(x) - Q(x))}{x^{\theta}} = +\infty
\end{equation}

\noindent
and 

\begin{equation}\label{desired-2}
\liminf\limits_{x\to \infty} \frac{Re(A(x) - Q(x))}{x^{\theta}} = -\infty,
\end{equation}
where $\theta = \frac{\delta}{2} - \frac{1}{4M}$.
\end{theorem}

\medskip
\noindent
If we replace $Re(b_{n_k})$ by $Im(b_{n_k})$, then we obtain analogous statements for $Im(A(x) - Q(x))$ in place of $Re(A(x) - Q(x))$.

\bigskip

\section{Statements of Theorems}

\noindent
In this section we state our results. First, we generalize Theorem \ref{l3} to finitely many number fields and our precise statement is as follows.

\begin{theorem}\label{main-th}
Let $l \geq 1$ and $k_1, \ldots , k_l \geq 2$ be integers. Let $K_1, \ldots , K_{l-1} \mbox{ and } K_l$ be finite Galois extensions of $\mathbb{Q}$ with discriminants $D_1, \ldots , D_{l-1} \mbox{ and } D_{l}$, respectively. Suppose that $[K_i : \mathbb{Q}]=d_i \geq 3$ is odd for all $i=1,\ldots , l$ and $\gcd(D_i, D_j)=1$ for $i \neq j$. Then for any $\epsilon > 0$, we have 
\begin{equation}\label{main-eq}
\displaystyle\sum_{n \leq x}\tau^{K_1}_{k_1}(n^2)\ldots \tau^{K_l}_{k_l}(n^2)=xP_{m}(\log x)+O(x^{1-\frac{3}{md_1\ldots d_l+6}+\epsilon}), 
\end{equation}
where $m=\displaystyle\prod_{i=1}^{l}\frac{k_{i}^2d_i + k_i}{2}$ and $P_m$ stands for a polynomial of degree $m-1$.
\end{theorem}

\vspace{2mm}

\noindent
We note that in Theorem \ref{main-th}, we are required to have $d_i$ odd for each $i$. But if we restrict ourselves to two number fields, one among them being quadratic, then we can have a similar result for the mean values of the ideal counting function $a_{K}(n)$. In \cite{yang}, Z. Yang found an asymptotic formula for the product of the ideal counting functions of two quadratic fields taken together (see also \cite{lu2}). In the next theorem of ours, we consider the same problem for a quadratic field and an odd degree number field. More precisely, the statement is as follows.

\begin{theorem}\label{quadratic-cubic}
Let $K_1$ and $K_2$ be two number fields with discriminants $D_{K_1}$ and $D_{K_2}$, respectively such that $[K_1 : \mathbb{Q}]=2$ and $[K_2 : \mathbb{Q}]=d$ with $d \geq 3$ odd. Suppose that $\gcd(D_{K_1}, D_{K_2})=1$. Then for any integer $l \geq 1$, we have $$\displaystyle\sum_{n \leq x} a_{K_1}(n)^{l}a_{K_2}(n)^{l} = xP_{m}(\log x) + O\left(x^{1-\frac{3}{(2d)^l + 6}+\epsilon}\right).$$
\end{theorem}

\vspace{2mm}

\noindent
We now turn our attention towards the error term arising in Theorem \ref{main-th}. Keeping the same notations, if we define $\Delta (x):=\displaystyle\sum_{n \leq x}\tau^K_{k_1}(n^2)\ldots \tau^K_{k_l}(n^2)-xP_{m}(\log x)$, then we get $$\Delta (x) = O(x^{1-\frac{3}{md_1\ldots d_l+6}+\epsilon}).$$

\vspace{2mm}

\noindent
One of the useful ways to understand $\Delta (x)$ is via the higher order moments. Recently, in \cite{shi}, S. Shi obtained the second order moment of the error term arising in the estimation of $\displaystyle\sum_{n \leq x}\tau^K_{k_1}(n)\ldots \tau^K_{k_l}(n)$. Here, adapting the methods and techniques used in \cite{shi}, we calculate the second order moment for the error term $\Delta (x)$ appearing in Theorem \ref{main-th}. More precisely, we prove the following.


\begin{theorem}\label{main-th-error}
Under the same notations and hypotheses as in Theorem \ref{main-th}, let $$\Delta (x):=\displaystyle\sum_{n \leq x}\tau^{K_1}_{k_1}(n^2)\ldots \tau^{K_l}_{k_l}(n^2)-xP_{m}(\log x).$$ Then for a given $\epsilon > 0$, we have 
\begin{equation}
\displaystyle\int_{1}^{X}\Delta^2 (x)dx = O(X^{3-\frac{6}{md_1\ldots d_l+3}+\epsilon}).
\end{equation}
Here, the implied constant depends only on $\epsilon$.
\end{theorem}

\noindent
Now, for a positive integer $a$, we define the {\it generalized sum of divisor function}  in a number field $K$ by the formula $$\sigma_{a}^{K}(n)= \displaystyle\sum_{N(\mathfrak{a}) \mid n}(N(\mathfrak{a}))^{a}.$$

\noindent
The precise statements of our results regarding the estimation of the summatory function are  given below.

\begin{theorem}\label{funccc}
Let $K$ be a number field of degree $d$. Then for any positive integer $a\geq 1$ and any $\epsilon>0$, we have
\begin{align*} \label{sum of divisor  func}
\displaystyle\sum_{n\leq x}\sigma_{a}^K (n)=\frac{c_K\zeta(1+a)}{1+a}x^{1+a}+O\left(x^{(1+a)-\frac{3}{d+6}+\epsilon}\right),
\end{align*}
where $c_K$ is the residue of $\zeta_K(s)$ at $s=1.$
\end{theorem}

\begin{theorem} \label{sum-of-divisor-function}
Let $K_1$ and $K_2$ be two number fields of degree $d_1$ and $d_2$ and discriminant $D_{K_1}$ and $D_{K_2}$, respectively. Suppose $K_1$ and $K_2$ satisfy the following conditions.

\begin{itemize}
\item[(i)] Both $K_1$ and $K_2$ are Galois over $\mathbb{Q}$;


\item[(ii)] $\gcd(D_{K_1}, D_{K_2}) = 1$.

\end{itemize}

\noindent
Then for positive integers $a$, $b$ and any $\epsilon > 0$, we have, $$\displaystyle\sum_{n \leq x}\sigma_{a}^{K_1}(n) \sigma_{b}^{K_2}(n) = \frac{\zeta(1+a+b)\zeta_{K_1}(1+b)\zeta_{K_2}(1+a)}{c_{K_1K_2}^{-1}(1+a+b)}x^{1+a+b}+O\left(x^{(1+a+b)-\frac{3}{d_1d_2 +6}+\epsilon}\right),$$ where $c_{K_1K_2}$ is the residue of $\zeta_{K_1K_2}(s)$ at $s=1.$

\end{theorem}

\begin{remark}
Theorem \ref{sum-of-divisor-function} can be generalized for any finite number of number fields with pairwise coprime discriminants.
\end{remark}

\noindent
Finally, as  an application of Theorem \ref{omega-type-theorem-kc-rn}, we get the following corollary.

\begin{corollary} \label{corollary 1}
Let $K$ be a quadratic field with discriminant $D_K$. Then we have

\begin{equation*}
\limsup\limits_{x\to \infty} \frac{\displaystyle\sum_{n \leq x}{\tau_{k} ^{K}(n) - xP_{k}(\log x)}}{x^{\frac{1}{2} - \frac{1}{4k}}} = +\infty,
\end{equation*}

\noindent
and 

\begin{equation*}
\liminf\limits_{x\to \infty} \frac{\displaystyle\sum_{n \leq x}{\tau_{k} ^{K}(n) - xP_{k}(\log x)}}{x^{\frac{1}{2} - \frac{1}{4k}}} = -\infty,
\end{equation*}
where $P_k$ is a polynomial of degree $k-1$ appearing in the main term in the asymptotic formula of $\displaystyle\sum_{n \leq x}\tau _{k}^{K}(n)$.
\end{corollary}

\begin{remark}
For the function $\sigma_{a}^{K}(n)$, the functional equation exists for the corresponding $L$-function $\zeta (s)\cdot \zeta_{K}(s-a)$ but it does not satisfy equation \eqref{omega-type-eq-2}. Therefore, corresponding results similar to \eqref{desired-1} and \eqref{desired-2} cannot be obtained using Theorem \ref{omega-type-theorem-kc-rn}.
\end{remark}

\begin{remark}
In this paper, for most of the cases we do not have the functional equation for the $L$-series corresponding to the summatory function of a given arithmetic function. Instead of that, we get an expression of the form $L(s) = \zeta _{k} ^{K}(s) \cdot U(s)$ for some Dirichlet series $U(s)$ which is absolutely convergent for $Re(s) > \frac{1}{2}$. Therefore, it is quite difficult to obtain expressions similar to \eqref{desired-1} and \eqref{desired-2}.
\end{remark}

\medskip

\section{Preliminaries}

\noindent
In this section, we list all the necessary results required for the proofs of Theorem \ref{main-th}-Theorem \ref{sum-of-divisor-function}. For that, let us start with a few fundamental properties of the Dedekind zeta-function associated to a number field. For an algebraic number field $K$, the Dedekind zeta-function of $K$ is defined by $$\zeta_K (s)=\displaystyle\sum_{\mathfrak{a} \subseteq \mathcal{O}_K}\frac{1}{N(\mathfrak{a})^s} \mbox{ for } Re(s) > 1,$$ where the sum runs over all the non-zero ideals in the ring $\mathcal{O}_K$ and the series is absolutely convergent in the half-plane $Re(s) > 1$. When we expand the above expression in the form of a Dirichlet series, it has the following form.

\begin{equation}\label{zeta-fn}
\zeta_K (s)=\displaystyle\sum_{\mathfrak{a} \subseteq \mathcal{O}_K}\frac{1}{N(\mathfrak{a})^s}=\displaystyle\sum_{n=1}^{\infty}\frac{a_{K}(n)}{n^s} \mbox{ for } Re(s) > 1.
\end{equation}
Here, $a_{K}(n)$ denotes the number of ideals in $\mathcal{O}_K$ having norm $n$, often called the {\it ideal counting function} of the number field $K$. Our first lemma asserts that $a_{K}(n)$ is a multiplicative function of $n$ and the proof can be found in \cite{kc2}.

\begin{lemma} \cite{kc2}\label{lemma-mult}
For a number field $K$, let $a_{K}(n)$ be the number of ideals in $\mathcal{O}_K$ with norm $n$. Then $a_{K}(n)$ is a multiplicative function of $n$ and for any $\epsilon > 0$, 
\begin{equation*}\label{coeff-multiplicative}
a_{K}(n) \ll_{\epsilon} n^{\epsilon} \mbox{ holds }.
\end{equation*}
\end{lemma}

\vspace{2mm}

\noindent
Since the coefficients of the Dirichlet series \eqref{zeta-fn} are multiplicative, $\zeta_{K}(s)$ admits the following Euler product expansion.

\begin{equation*}\label{Euler-product-expansion}
\zeta_K (s)=\displaystyle\sum_{n=1}^{\infty}\frac{a_{K}(n)}{n^s}=\displaystyle\prod_{p}\left(1 + \frac{a_p}{p^s} + \frac{a_{p^2}}{p^{2s}} +\ldots \right) \mbox{ for } Re(s) > 1.
\end{equation*}

\vspace{2mm}

\noindent
As we are dealing with several number fields at a time, it will be quite useful to know the behaviour of the ideal counting function at prime arguments in the compositum of two number fields. 
\begin{lemma} \cite{lu-ma} \label{compositum-lemma}
Let $K$ and $L$ be two number fields with discriminants $D_K$ and $D_L$, respectively. Suppose that $\gcd(D_K, D_L)=1$. Then for any prime number $p$, we have $$a_{KL}(p)=a_K(p)a_L(p).$$
\end{lemma}

\vspace{2mm}

\noindent
The next couple of lemmas, which relate the divisor function of a finite Galois extension to the ideal counting function, will play a crucial role in the course of the proofs of our theorems.

\begin{lemma} \cite{lu3}\label{divisor-ideal}
Let $k \geq 2$ be an integer and let $K/\mathbb{Q}$ be a finite Galois extension of odd degree $d$. Then $$\tau_{k}^{K}(p^2)=\frac{k^2d + k}{2}a_{K}(p)$$ holds for all but finitely many prime numbers $p$.
\end{lemma}

\begin{lemma} \cite{lu1} \label{lu-acta-math-hungarica}
Let $K/\mathbb{Q}$ be finite Galois extension of degree $d$. Then for any positive integer $k$, the relation $$a_{K}(p)^{k} = d^{k-1}a_{K}(p)$$ holds for all but finitely many prime numbers $p$.
\end{lemma}

\vspace{2mm}

\noindent
Now, we state the {\it Phragman-Lindel\"{o}ff} hypothesis which deals with the estimation of an analytic function in a given strip.
\begin{lemma} \cite{kowalski} \label{phragman-lindeloff}
Let $a, b$ be two real numbers with $a < b$ and let $f$ be an analytic function defined on an open neighbourhood of the strip $a \leq \sigma \leq b$, for some real numbers $a < b$, such that $|f(s)|=O(e^{|s|^A})$ for some $A \geq 0$ and for all $s$ satisfying $a \leq Re(s) \leq b$.
\begin{itemize}
\item[(i)] Assume that $|f(s)| \leq M \mbox{ for all } s$ on the boundary of the critical strip, i.e, for $\sigma = a$ or $\sigma = b$. Then we have $|f(s)| \leq M$ for all $s$ in the strip.

\item[(ii)] If there exist real numbers $M_a, M_b, \alpha, \beta$ such that for all $t \in \mathbb{R}$,  $$|f(a+it)|\leq M_{a}(1+|t|)^{\alpha}$$ and $$|f(b+it)|\leq M_{b}(1+|t|)^{\beta},$$ 
\end{itemize}
then for all $s=\sigma + it$ in the strip $a \leq Re(s) \leq b$, we have $$|f(\sigma + it)|\leq M_{a}^{\ell(\sigma)}M_{b}^{1-\ell(\sigma)}(1+|t|)^{\alpha \ell(\sigma)+\beta (1-\ell(\sigma))},$$ where $\ell$ is the linear function satisfying $\ell (a)=1$ and $\ell (b)=0$.
\end{lemma}

\vspace{2mm}

\noindent
Next, we state and prove a result related to two Dirichlet series in a general set-up, which is a mild generalization of a result stated in \cite{kc1}.

\begin{lemma}\label{Dirichlet-series-comparison}
Let $f(s)=\displaystyle\sum_{n=1}^{\infty}\frac{a_n}{n^s}$ and $g(s)=\displaystyle\sum_{n=1}^{\infty}\frac{b_n}{n^s}$ be two Dirichlet series both of which are absolutely convergent in $Re(s) > 1$ and satisfy the following conditions:
\begin{itemize}
\item[(i)] Both $a_n$ and $b_n$ are positive and multiplicative functions of $n$,

\item[(ii)] For any $\epsilon > 0$, we have $a_n \ll_{\epsilon} n^{\epsilon}$ and $b_n \ll_{\epsilon} n^{\epsilon}$,

\item[(iii)] $a_p = b_p$ for all but finitely many prime numbers $p$.
\end{itemize}

\noindent
Then $f(s)=g(s)\cdot U(s)$, where $U(s)$ is a Dirichlet series which is absolutely convergent in $Re(s) > \frac{1}{2}$ and uniformly convergent in $Re(s) \geq \frac{1}{2} + \epsilon$ for any $\epsilon > 0$.
\end{lemma}

\begin{proof}
Since both $a_n$ and $b_n$ are multiplicative functions of $n$, $f(s)$ and $g(s)$ admit Euler product expansions. For $Re(s) = \sigma > 1$, consider $f(s)\cdot g(s)^{-1} = \displaystyle\prod_{p} U_{p}(s)$, where $$U_{p}(s) = \left(1 + \frac{a_p}{p^s} + \displaystyle\sum_{m = 2}^{\infty} \frac{a_{p^m}}{p^{ms}}\right)\cdot \left(1 + \frac{b_p}{p^s} + \displaystyle\sum_{m = 2}^{\infty} \frac{b_{p^m}}{p^{ms}}\right)^{-1}.$$ 

\noindent
Let $U_{p}^{\prime}(s)=\displaystyle\sum_{m = 2}^{\infty} \frac{a_{p^m}}{p^{ms}}$ and $U_{p}^{\prime \prime}(s)=\displaystyle\sum_{m = 2}^{\infty} \frac{b_{p^m}}{p^{ms}}$. Then using condition (ii), we get 
$$
U_{p}^{\prime}(s) \ll_{\epsilon} \displaystyle\sum_{m = 2}^{\infty} \frac{p^{m\epsilon}}{p^{m\sigma}} = \frac{1}{p^{\sigma - \epsilon}}\cdot \frac{1}{p^{(\sigma - \epsilon)}-1} \ll_{\epsilon} \frac{1}{p^{2\sigma - 2\epsilon}}.
$$

\noindent
Similarly, $U_{p}^{\prime \prime}(s) \ll_{\epsilon} \frac{1}{p^{2\sigma - 2\epsilon}}$. Hence we obtain 
\begin{equation}\label{eqqqqqqqqqqn}
U_{p}(s) = \Bigg(1 + \frac{a_p}{p^s} + O_{\epsilon}\Bigg(\frac{1}{p^{2\sigma - 2\epsilon}}\Bigg)\Bigg)\cdot \Bigg(1 + \frac{b_p}{p^s} + O_{\epsilon}\Bigg(\frac{1}{p^{2\sigma - 2\epsilon}}\Bigg)\Bigg)^{-1}.
\end{equation}

\noindent
Now from \eqref{eqqqqqqqqqqn} and the hypotheses (i) and (iii), we conclude that $U_{p}(s) = 1 + O(p^{\epsilon - 2\sigma})$ for all but finitely many primes $p$.

\noindent
Hence $\displaystyle\prod_{p} U_{p}(s)$ converges if and only if $\displaystyle\sum_{p}\frac{1}{p^{2\sigma - \epsilon}}$ converges. In other words, for $\sigma > \frac{1}{2}$, $U(s) = \displaystyle\prod_{p}U_{p}(s)$ defines an absolutely convergent Dirichlet series and therefore we obtain $f(s)=g(s)\cdot U(s)$.
\end{proof}

\vspace{2mm}

\noindent
Now, we record the following estimate, which can be seen in \cite{heath}, regarding the growth of the Dedekind zeta-function of a number field on the half-line $Re(s)=\frac{1}{2}$.
\begin{lemma} \cite{heath} \label{heath-brown}
Let $K$ be a number field of degree $d$. Then for any real number $t \geq 1$ and $\epsilon > 0$, we have $$\zeta_{K}\left(\frac{1}{2}+it\right)=O\left(t^{\frac{d}{6}+\epsilon}\right),$$ where the implied constant depends only on the number field $K$.
\end{lemma}

\noindent
Lastly, the following lemma provides the Perron's formula in general set-up.

\begin{lemma}  [\cite{ivic},  Equation $(A.10)$ ] \label{perron}
Let $a(n)$ be an arithmetic function satisfying $|a(n)| \ll \Phi(n)$, where $\Phi$ is an increasing function of $n$ and 
\begin{align*}
\displaystyle\sum_{n=1}^{\infty}|a(n)|n^{-\sigma} \ll (\sigma -1)^{-\alpha}
\end{align*} 
as $\sigma \to 1^{+}$ and for some real number $\alpha > 0$. 

\noindent
Let $1 < b \ll 1$, $T > 0$ and $x \geq 1$.  Let us define  $F(s) = \displaystyle\sum_{n=1}^{\infty}\frac{a(n)}{n^s}$. Then we have 

\begin{equation*}
\displaystyle\sum_{n \leq x} a(n)= \frac{1}{2 \pi i} \displaystyle\int_{b-iT}^{b+iT} F(s)\frac{x^s}{s}ds + O\left(x^b T^{-1}(b-1)^{-\alpha}\right) + O\left(xT^{-1}\Phi (2x)\log (2x)\right) + O\left(\Phi (2x)\right).
\end{equation*}
\end{lemma}

\section{Proof of Theorem \ref{main-th}}
\noindent
Our strategy, to prove Theorem \ref{main-th}, is via studying the Dedekind zeta-function for a suitable number field and its connection to a particular Dirichlet series. For integers $l \geq 1$ and $k_1, \ldots , k_l \geq 2$, we define  
\begin{equation}\label{l-series}
L_{k_{1},\ldots ,k_{l}}^{K_{1},\ldots ,K_{l}}(s)=\displaystyle\sum_{n=1}^{\infty}\frac{\tau_{k_1}^{K_1}(n^2)\ldots \tau_{k_l}^{K_l}(n^2)}{n^s} \mbox{ for } Re(s) > 1.
\end{equation}

\vspace{2mm}

\noindent
Using the well known result $\tau(n)=O(n^{\epsilon})$ for any $\epsilon > 0$, we get that for any integer $k \geq 1$ and any number field $K$, $$\tau_{k}^{K}(n)=O\left(\displaystyle\sum_{n=n_{1}\ldots n_{k}}\tau_{k}(n_1) \ldots \tau_{k}(n_k)\right)=O(n^{\epsilon}),$$ which also implies that $\tau_{k}^{K}(n^2)=O(n^{\epsilon})$. Thus the Dirichlet series defined in \eqref{l-series} is absolutely convergent in the half-plane $Re(s) > 1$.

\vspace{2mm}

\noindent
Since each $\tau_{k_i}^{K_i}(n)$ is multiplicative, so are the coefficients $\tau_{k_1}^{K_1}(n^2)\ldots \tau_{k_l}^{K_l}(n^2)$ of $n^{-s}$ in \eqref{l-series}. Thus in $Re(s) > 1$, by using Lemma \ref{compositum-lemma} and Lemma \ref{divisor-ideal}, we have 
\begin{align*}
L_{k_{1},\ldots ,k_{l}}^{K_{1},\ldots ,K_{l}}(s) &=\displaystyle\sum_{n=1}^{\infty}\frac{\tau_{k_1}^{K_1}(n^2)\ldots \tau_{k_l}^{K_l}(n^2)}{n^s}\\ &= \displaystyle\prod_{p}\left(1 + \frac{\tau_{k_1}^{K_1}(p^2)\ldots \tau_{k_l}^{K_l}(p^2)}{p^s} + \frac{\tau_{k_1}^{K_1}(p^4)\ldots \tau_{k_l}^{K_l}(p^4)}{p^{2s}} + \ldots \right)\\ &= \displaystyle\prod_{p}\Bigg(1 + \frac{\displaystyle\prod_{i=1}^{l}\left(\frac{k_{i}^2d_{i} + k_i}{2}\right)a_{K_{1}\ldots K_{l}}(p)}{p^s} + \ldots\Bigg)
\end{align*}
for all but finitely many primes.

\vspace{2mm}

\noindent
On the other hand, if we consider the Dedekind zeta-function for the compositum of the number fields $K_1, \ldots , K_l$, then we get for $Re(s) > 1$,
\begin{equation*}\label{zeta-power}
(\zeta_{K_1\ldots K_l}(s))^{\displaystyle\prod_{i=1}^{l}\left(\frac{k_{i}^2d_i + k_i}{2}\right)}=\displaystyle\prod_{p}\left(1 + \frac{\displaystyle\prod_{i=1}^{l}\left(\frac{k_{i}^2d_{i} + k_i}{2}\right)a_{K_{1}\ldots K_{l}}(p)}{p^s} + \ldots\right)
\end{equation*}

\vspace{2mm}

\noindent
Hence we see that the coefficients of $p^{-s}$ in the Euler product expansions of $L_{k_{1},\ldots ,k_{l}}^{K_{1},\ldots ,K_{l}}(s)$ and $(\zeta_{K_1\ldots K_l}(s))^{\displaystyle\prod_{i=1}^{l}\left(\frac{k_{i}^2d_i + k_i}{2}\right)}$ are same for all but finitely many prime numbers $p$. Thus by Lemma \ref{Dirichlet-series-comparison}, we have 
\begin{equation}\label{sei-equation}
L_{k_{1},\ldots ,k_{l}}^{K_{1},\ldots ,K_{l}}(s)=(\zeta_{K_1\ldots K_l}(s))^{\displaystyle\prod_{i=1}^{l}\left(\frac{k_{i}^2d_i + k_i}{2}\right)}\cdot U(s),
\end{equation} 
where $U(s)$ is a Dirichlet series, absolutely convergent in $Re(s) > \frac{1}{2}$ and uniformly convergent in $Re(s) \geq \frac{1}{2} + \epsilon$, for any $\epsilon > 0$. From \eqref{sei-equation}, we can say that $L_{k_{1},\ldots ,k_{l}}^{K_{1},\ldots ,K_{l}}(s)$ has an analytic continuation to the half-plane $Re(s) > \frac{1}{2}$, except for a pole at $s=1$ of order ${\displaystyle\prod_{i=1}^{l}\left(\frac{k_{i}^2d_i + k_i}{2}\right)}$. Now using Lemma \ref{perron}, we get 

\begin{equation}\label{using-perron}
\displaystyle\sum_{n \leq x}\tau_{k_1}^{K_1}(n^2)\ldots \tau_{k_l}^{K_l}(n^2)=\frac{1}{2\pi i}\displaystyle\int_{b-iT}^{b+iT}L_{k_{1},\ldots ,k_{l}}^{K_{1},\ldots ,K_{l}}(s)\frac{x^s}{s}ds + O\left(\frac{x^{1 + \epsilon}}{T}\right),
\end{equation}
where $b=1 + \epsilon$ and $T$ is a parameter satisfying $1\leq T \leq x$, to be appropriately chosen later.

\vspace{2mm}

\noindent
Now, we move the line of integration to $Re(s)=\frac{1}{2} + \epsilon$, and consider the rectangular contour formed by joining the points $\frac{1}{2}+\epsilon - iT, \frac{1}{2}+\epsilon + iT, b+iT \mbox{ and } b-iT$, successively. From \eqref{using-perron}, by using Cauchy residue theorem, we get
\begin{eqnarray*}\label{sei-estimation}
\displaystyle\sum_{n \leq x}\tau_{k_1}^{K_1}(n^2)\ldots \tau_{k_l}^{K_l}(n^2) &=& \frac{1}{2\pi i}\left[\displaystyle\int_{\frac{1}{2}+\epsilon -iT}^{\frac{1}{2}+\epsilon +iT} + \displaystyle\int_{\frac{1}{2}+\epsilon +iT}^{b + iT} + \displaystyle\int_{b-iT}^{\frac{1}{2}+\epsilon - iT}\right]L_{k_{1},\ldots ,k_{l}}^{K_{1},\ldots ,K_{l}}(s)\frac{x^s}{s}ds \\ &+& \displaystyle Res_{s=1}L_{k_{1},\ldots ,k_{l}}^{K_{1},\ldots ,K_{l}}(s)\frac{x^s}{s} + O\left(\frac{x^{1+\epsilon}}{T}\right) \nonumber \\ &=& xP_{m}(\log x) + I_1 + I_2 + I_3 + O\left(\frac{x^{1+\epsilon}}{T}\right) \nonumber ,
\end{eqnarray*}
where $I_1=\frac{1}{2\pi i}\displaystyle\int_{\frac{1}{2}+\epsilon -iT}^{\frac{1}{2}+\epsilon +iT}L_{k_{1},\ldots ,k_{l}}^{K_{1},\ldots ,K_{l}}(s)\frac{x^s}{s}ds,\qquad I_2=\frac{1}{2\pi i}\displaystyle\int_{\frac{1}{2}+\epsilon +iT}^{b +iT}L_{k_{1},\ldots ,k_{l}}^{K_{1},\ldots ,K_{l}}(s)\frac{x^s}{s}ds$,\\ $I_3=\frac{1}{2\pi i}\displaystyle\int_{b -iT}^{\frac{1}{2}+\epsilon -iT}L_{k_{1},\ldots ,k_{l}}^{K_{1},\ldots ,K_{l}}(s)\frac{x^s}{s}ds$, \qquad $m=\displaystyle\prod_{i=1}^{l}\left(\frac{k_{i}^2d_{i} + k_{i}}{2}\right)$ and $P_m$ is a polynomial of degree $m-1$.

\vspace{2mm}

\noindent
Now, by using Lemma \ref{phragman-lindeloff} and Lemma \ref{heath-brown}, we get 
\begin{equation}\label{zeta-power-estimate}
\zeta_{K_1\ldots K_l}(\sigma + it) \ll (1+|t|)^{\frac{d_1\ldots d_l}{3}(1-\sigma)+\epsilon}.
\end{equation}

\vspace{2mm}

\noindent
Raising both sides of \eqref{zeta-power-estimate}
\ to $m$-th power, we get 
\begin{equation}\label{zeta-power-estimate-1}
\zeta^{m}_{K_1\ldots K_l}(\sigma + it) \ll (1+|t|)^{\frac{md_1\ldots d_l}{3}(1-\sigma)+\epsilon}.
\end{equation}

\vspace{2mm}

\noindent
Since $U(s)$ is absolutely convergent in the region $Re(s) > \frac{1}{2}$, it is bounded in that region. Thus by \eqref{sei-equation} and \eqref{zeta-power-estimate-1}, we get
\begin{eqnarray*}
I_1 &=& \frac{1}{2\pi i}\displaystyle\int_{\frac{1}{2}+\epsilon -iT}^{\frac{1}{2}+\epsilon +iT}L_{k_{1},\ldots ,k_{l}}^{K_{1},\ldots ,K_{l}}(s)\frac{x^s}{s}ds\\ &\ll& \displaystyle\int_{-T}^{T}\Bigg|L_{k_{1},\ldots ,k_{l}}^{K_{1},\ldots ,K_{l}}\Bigg(\frac{1}{2}+\epsilon +it\Bigg)\Bigg|\cdot \Bigg|\frac{x^{\frac{1}{2}+\epsilon +it}}{\frac{1}{2}+\epsilon +it}\Bigg|dt\\ &\ll& x^{\frac{1}{2}+\epsilon} + x^{\frac{1}{2}+\epsilon}\Bigg(\displaystyle\int_{-T}^{-1}\Bigg|\frac{L_{k_{1},\ldots ,k_{l}}^{K_{1},\ldots ,K_{l}}(\frac{1}{2}+\epsilon +it)}{t}\Bigg|dt + \displaystyle\int_{1}^{T}\Bigg|\frac{L_{k_{1},\ldots ,k_{l}}^{K_{1},\ldots ,K_{l}}(\frac{1}{2}+\epsilon +it)}{t}\Bigg|dt \Bigg)\\ &\ll& x^{\frac{1}{2}+\epsilon} + x^{\frac{1}{2}+\epsilon}\Bigg(\displaystyle\int_{1}^{T}\Bigg(\frac{|L_{k_{1},\ldots ,k_{l}}^{K_{1},\ldots ,K_{l}}(\frac{1}{2}+\epsilon - it)|+|L_{k_{1},\ldots ,k_{l}}^{K_{1},\ldots ,K_{l}}(\frac{1}{2}+\epsilon + it)|}{t}\Bigg)dt\Bigg)\\ &\ll& x^{\frac{1}{2} + \epsilon} + x^{\frac{1}{2} + \epsilon}\Bigg(\displaystyle\int_{1}^{T}\Bigg(\left|\zeta_{K_1\ldots K_l} \left(\frac{1}{2} + \epsilon - it\right)\right|^m + \Bigg|\zeta_{K_1\ldots K_l} \left(\frac{1}{2} + \epsilon + it\right)\Bigg|^m\Bigg)t^{-1}dt\Bigg)\\ &\ll& x^{\frac{1}{2}+\epsilon} + x^{\frac{1}{2} + \epsilon} \Bigg(\displaystyle\int_{1}^{T}t^{\frac{md_1\ldots d_l}{6}+\epsilon -1}dt\Bigg)\\ &\ll& x^{\frac{1}{2} + \epsilon} + x^{\frac{1}{2} + \epsilon}T^{\frac{md_1 \ldots d_l}{6} + \epsilon}
\end{eqnarray*}

\noindent
For $j = 2 \mbox{ and } 3$, using \eqref{zeta-power-estimate-1}, we have

\begin{eqnarray*}
I_j &\ll& \displaystyle\int_{\frac{1}{2}+\epsilon}^{b}x^{\sigma}\left|\zeta_{K_1\ldots K_l}(\sigma + iT)\right|^mT^{-1}d\sigma\\ &\ll& \displaystyle\max_{\frac{1}{2} + \epsilon \leq \sigma \leq b}x^{\sigma}T^{\frac{md_1 \ldots d_l (1-\sigma)}{3} + \epsilon}T^{-1}\\ &\ll& \displaystyle\max_{\frac{1}{2} + \epsilon \leq \sigma \leq b}\Bigg(\frac{x}{T^{\frac{md_{1}\ldots d_{l}}{3}}}\Bigg)^{\sigma}T^{\frac{md_{1}\ldots d_{l}}{3}-1+\epsilon}\\ &\ll& \frac{x^{1+\epsilon}}{T} + x^{\frac{1}{2}+\epsilon}T^{\frac{md_1 \ldots d_l}{6}-1+\epsilon}.
\end{eqnarray*}

\vspace{2mm}

\noindent
Combining the above estimations, we obtain $$\displaystyle\sum_{n \leq x}\tau^{K_1}_{k_1}(n^2)\ldots \tau^{K_l}_{k_l}(n^2)=xP_{m}(\log x)+O\left(\frac{x^{1+\epsilon}}{T}\right) + O\left(x^{\frac{1}{2}+\epsilon}T^{\frac{md_1 \ldots d_l}{6}+\epsilon}\right).$$

\noindent
By taking $T=x^{\frac{3}{md_1 \ldots d_l + 6}}$, we get $$\displaystyle\sum_{n \leq x}\tau^{K_1}_{k_1}(n^2)\ldots \tau^{K_l}_{k_l}(n^2)=xP_{m}(\log x)+O(x^{1-\frac{3}{md_1\ldots d_l+6}+\epsilon}).$$ This completes the proof of Theorem \ref{main-th}. $\hfill\Box$

\medskip

\section{Proof of Theorem \ref{quadratic-cubic}}

\noindent
We consider the Dirichlet series, defined in 

\begin{equation*}\label{equation-for-quadratic-cubic}
L_{K_1 , K_2}(s) = \displaystyle\sum_{n=1}^{\infty}\frac{a_{K_1}(n)^l a_{K_2}(n)^l}{n^s} \mbox{ for } Re(s) > 1.
\end{equation*}

\noindent
Since $a_{K_1}(n)^l a_{K_2}(n)^l$ is a multiplicative function of $n$, using Lemma \ref{compositum-lemma}, we get the Euler product expansion of $L_{K_1 , K_2}(s)$ as $\displaystyle\prod_{p}\left(1 + \frac{a_{K_1 K_2} (p)^l}{p^s} + \ldots\right)$ for all but finitely many primes $p$. Comparing this with $\zeta_{K_1 K_2}(s)^{(2d)^{l-1}}$, we see that the coefficients of $p^{-s}$ in the Euler product of these two Dirichlet series are same for all but finitely many prime numbers $p$. Thus by Lemma \ref{Dirichlet-series-comparison}, we have $$L_{K_1 , K_2}(s) = \zeta_{K_1 K_2}(s)^{(2d)^{l-1}}\cdot U(s)$$ where $U(s)$ is a Dirichlet series, absolutely convergent for $Re(s) > \frac{1}{2}$. The rest of the proof follows exactly the same line of argument as that of Theorem \ref{main-th} and hence we omit it. $\hfill\Box$

\section{Proof of Theorem \ref{main-th-error}}

From equation \eqref{sei-estimation}, we have
\begin{align*}
\displaystyle\sum_{n \leq x}\tau_{k_1}^{K_1}(n^2)\ldots \tau_{k_l}^{K_l}(n^2)= xP_{m}(\log x) + I_1 + I_2 + I_3 + O\left(\frac{x^{1+\epsilon}}{T}\right).  
\end{align*}

In order to prove Theorem \ref{main-th-error}, we need to find upper bounds of the following integrals.

\begin{align*}
\displaystyle\int_{1}^{X}I_j^2(x)dx, \quad j=1, 2, 3 \quad \text{and} \quad \displaystyle\int_{1}^{X} \left(O\left(\frac{x^{1+\epsilon}}{T}\right)\right)^2 dx.
\end{align*}

It is easy to see that,
\[
\displaystyle\int_{1}^{X} \left(O\left(\frac{x^{1+\epsilon}}{T}\right)\right)^2 dx
\ll \frac{X^{3+\epsilon}}{T^2}.
\]

We also need to use the following inequality to estimate $\displaystyle\int_{1}^{X}I^{2}_{1}(x)dx$ .
\begin{align}\label{integral_bound}
\displaystyle\int_{-T}^{T}\frac{dt_2}{1+|t_1-t_2|}\ll \log (2T).
\end{align}

Using inequalities \eqref{sei-equation}, \eqref{zeta-power-estimate-1} and \eqref{integral_bound}, we get
\begin{align*}
&\displaystyle\int_{1}^{X}I_1^2(x)dx= \frac{1}{4\pi^2} \displaystyle\int_{1}^{X}\Bigg[\displaystyle\int_{-T}^{T}L_{k_{1},\ldots ,k_{l}}^{K_{1},\ldots ,K_{l}}\left(\frac{1}{2}+\epsilon+it_1\right)\frac{x^{1/2+\epsilon+it_1}}{1/2+\epsilon+it_1}dt_1\\
&\times \displaystyle\int_{-T}^{T}\overline{L_{k_{1},\ldots ,k_{l}}^{K_{1},\ldots ,K_{l}}\left(\frac{1}{2}+\epsilon+it_2\right)}\frac{x^{1/2+\epsilon-it_2}}{1/2+\epsilon-it_2}dt_2\Bigg]dx\\
&=\frac{1}{4\pi^2}\displaystyle\int_{-T}^{T}\displaystyle\int_{-T}^{T}\frac{L_{k_{1},\ldots ,k_{l}}^{K_{1},\ldots ,K_{l}}\left(\frac{1}{2}+\epsilon+it_1\right)\overline{L_{k_{1},\ldots ,k_{l}}^{K_{1},\ldots ,K_{l}}\left(\frac{1}{2}+\epsilon+it_2\right)}}{(1/2+\epsilon+it_1)(1/2+\epsilon-it_2)} \Big(\displaystyle\int_{1}^{X}x^{1+2\epsilon+i(t_1-t_2)}dx\Big) dt_1 dt_2 \\
& \ll X^{2+2\epsilon}\displaystyle\int_{-T}^{T}\displaystyle\int_{-T}^{T}\frac{\Big|L_{k_{1},\ldots ,k_{l}}^{K_{1},\ldots ,K_{l}}\left(\frac{1}{2}+\epsilon+it_1\right)\Big|\Big|L_{k_{1},\ldots ,k_{l}}^{K_{1},\ldots ,K_{l}}\left(\frac{1}{2}+\epsilon+it_2\right)\Big|}{(1+|t_1|)(1+|t_2|)(1+|t_1-t_2|)}dt_1 dt_2\\
\vspace{10 cm}& \ll X^{2+2\epsilon} \displaystyle\int_{-T}^{T}\displaystyle\int_{-T}^{T}\Bigg(\frac{\Big|L_{k_{1},\ldots ,k_{l}}^{K_{1},\ldots ,K_{l}}\left(\frac{1}{2}+\epsilon+it_1\right)\Big|^2}{(1+|t_1|)^2}+\frac{\Big|L_{k_{1},\ldots ,k_{l}}^{K_{1},\ldots ,K_{l}}\left(\frac{1}{2}+\epsilon+it_2\right)\Big|^2}{(1+|t_2|)^2}\Bigg)\frac{dt_1 dt_2}{1+|t_1-t_2|}\\
& \ll X^{2+2\epsilon} \log (2T)\displaystyle\int_{-T}^{T}\frac{\Big|L_{k_{1},\ldots ,k_{l}}^{K_{1},\ldots ,K_{l}}\left(\frac{1}{2}+\epsilon+it_1\right)\Big|^2}{(1+|t_1|)^2}dt_1\\
&\ll X^{2+2\epsilon}\log (2T)+X^{2+2\epsilon}\log (2T) \displaystyle\int_{1}^{T}\Big|\zeta^{m}_{K_1\ldots K_l}(1/2+\epsilon + it)\Big|^2\Big|U\left(1/2+\epsilon +it\right)\Big|^2 t^{-2}dt\\
& \ll X^{2+2\epsilon}\log (2T)+X^{2+2\epsilon}\log (2T) \displaystyle\int_{1}^{T}\left(t^{\frac{md_1\ldots d_l}{6}+\epsilon}\right)^2 t^{-2}dt\\
& \ll X^{2+2\epsilon}\log (2T)+X^{2+2\epsilon}\log (2T)T^{\frac{md_1\ldots d_l}{3}-1}.
\end{align*}

Again using inequalities \eqref{sei-equation} and \eqref{zeta-power-estimate-1}, for $j=2, 3$ we get
\begin{align*}
I_j(x)&\ll \displaystyle\int_{1/2+\epsilon}^{1+\epsilon}x^{\sigma}\big|\zeta^{m}_{K_1\ldots K_l}(\sigma + iT)\big| T^{-1}d\sigma \ll \max\limits_{1/2+\epsilon\leq \sigma\leq 1+\epsilon} x^{\sigma}T^{\frac{md_1\ldots d_l}{3}(1-\sigma)+\epsilon}T^{-1}\\
& \ll \frac{x^{1+\epsilon}}{T}+x^{1/2+\epsilon}T^{\frac{md_1\ldots d_l}{6}-1+\epsilon}.
\end{align*}

Therefore, for $j=2, 3$ we have
\[
\displaystyle\int_{1}^{X}I_j^2(x)dx\ll \frac{X^{3+\epsilon}}{T^2}+X^{2+2\epsilon}T^{\frac{md_1\ldots d_l}{3}-2+2\epsilon}.
\]

Let us choose, $T=X^{\frac{3}{md_1\ldots d_l+3}}.$ For this choice of $T$, from above calculations, we get

\[
\displaystyle\int_{1}^{X}\Delta^2 (x)dx = O(X^{3-\frac{6}{md_1\ldots d_l+3}+\epsilon}).
\] $\hfill\Box$

\section{Proof of Theorem \ref{funccc}}

\noindent
For $a\geq 1,$ we write $L_K(s)=\displaystyle\sum_{n = 1}^{\infty}\frac{\sigma_a^K(n)}{n^s}.$ From the definition of $\sigma_a^K(n)$, it is easy to see that $L_K(s)$ is convergent for $Re(s)>1+a.$

\noindent
Observe that, $L_K(s)=\zeta(s)\zeta_K(s-a).$ Now by using Lemma \ref{perron}, we get 

\begin{equation*}
\displaystyle\sum_{n \leq x}\sigma_a^K(n)=\frac{1}{2 \pi i}\displaystyle\int_{b-iT}^{b+iT}L_K(s)\frac{x^s}{s}ds + O\left(\frac{x^{1+a + \epsilon}}{T}\right),
\end{equation*}
where $b=1+a + \epsilon$ and $T$ is a parameter satisfying $1\leq T < x^{1+a}$ and to be appropriately chosen later.

\noindent
Now, we move the line of integration to $Re(s)=\frac{1}{2}+a $, and consider the rectangular contour formed by joining the points $\frac{1}{2}+a - iT, \frac{1}{2}+a + iT, b+iT \mbox{ and } b-iT$ successively. 
Since $\zeta(s)$ is bounded along the above vertical and horizontal line, so it is enough to consider the behaviour of Dedekind zeta function $\zeta_K(s-a).$ Also observe that, at $s=\frac{1}{2} +a+it$ the value of $\zeta_K(s-a)$ is $\zeta_K(\frac{1}{2}+it).$
\vspace{2mm}
 
\noindent
From this point onwards, by following the same argument as given in the proof Theorem \ref{main-th}, we get the desired result. $\hfill\Box$

\section{Proof of Theorem \ref{sum-of-divisor-function}}

\noindent
We consider $L(s)=\displaystyle\sum_{n=1}^{\infty}\frac{\sigma_{a}^{K_1}(n)\sigma_{b}^{K_2}(n)}{n^s}$. From the definition of $\sigma_{a}^{K_j}(n)$, $j=1, 2$ it is clear that $L(s)$ converges absolutely for $Re(s) > 1+a+b$.

\noindent
Now from Lemma \ref{compositum-lemma}, it easily follows that $\sigma_{a}^{K_1}(n)\sigma_{b}^{K_2}(n)$ is a multiplicative function of $n$. Also, we notice that for a number field $K$, 

\begin{equation*}\label{last-theorem-equation}
\sigma_{a}^{K}(p) = \displaystyle\sum_{N(\mathfrak{a}) \mid p} N(\mathfrak{a})^{a} = 1 + \displaystyle\sum_{N(\mathfrak{a})=p}p^a = 1 + p^{a} \displaystyle\sum_{N(\mathfrak{a}) = p}1 = 1 + a_{K}(p)\cdot p^{a}.
\end{equation*}

\noindent
Now, by Lemma \ref{lu-acta-math-hungarica} and the following equality 

\begin{equation*}
a_{K}(p)=\left\{\begin{array}{ll}
p; \mbox{ if } p \mbox{ splits completely }\\
0; \mbox{ otherwise }.
\end{array}\right.
\end{equation*}

\noindent
we get for $Re(s) > 1+a+b$

\begin{eqnarray*}
L(s) &=&\displaystyle\sum_{n=1}^{\infty}\frac{\sigma_{a}^{K_1}(n)\sigma_{b}^{K_2}(n)}{n^s}\\ &=& \displaystyle\prod_{p}\left(1 + \frac{\sigma_{a}^{K_1}(p)\sigma_{b}^{K_2}(p)}{p^s} + \ldots \right)\\ &=& \displaystyle\prod_{p}\left(1 + \frac{1 + p^{a}a_{K_1}(p) + p^{b}a_{K_2}(p) + p^{a+b} a_{K_1 K_2}(p)}{p^s} + \ldots\right)
\end{eqnarray*}
for all but finitely many prime numbers $p$.

\vspace{2mm}

\noindent
Now, comparing this with the Euler product expansion of $\zeta(s)\cdot \zeta_{K_1}(s-a)\cdot \zeta_{K_2}(s-b)\cdot \zeta_{K_1 K_2}(s-a-b)$, we see that the coefficients of $p^{-s}$ are same for all but finitely many primes $p$. Hence, by Lemma \ref{Dirichlet-series-comparison}, we have $L(s)=\zeta(s)\cdot \zeta_{K_1}(s-a)\cdot \zeta_{K_2}(s-b)\cdot \zeta_{K_1 K_2}(s-a-b)\cdot U'(s)$, where $U'(s)$ is a Dirichlet series, which is absolutely convergent for $Re(s) > \frac{1}{2}+a+b$.

\vspace{2mm}

\noindent
Again by following the same argument as given in the proof Theorem \ref{main-th}, we get our desired result. $\hfill\Box$

\vspace{2mm}

\section{Proof of Corollary \ref{corollary 1}}
\noindent
We know that the $L$-series corresponding to $\displaystyle\sum_{n \leq x} \tau_{k} ^{K}(n)$ is $\zeta_{K} ^{k}(s)$. Moreover, $\zeta_{K}(s)$ satisfies the functional equation $$\left(\frac{\sqrt{|D_K|}}{2\pi}\right)^s \Gamma (s) \zeta _{K} (s) = \left(\frac{\sqrt{|D_K|}}{2\pi}\right)^{1-s} \Gamma (1-s) \zeta _{K} (1-s) \mbox{ if } D_{K} < 0$$ and $$\left(\frac{\sqrt{|D_K|}}{\pi}\right)^s \Gamma^{2} \left(\frac{s}{2}\right) \zeta _{K} (s) = \left(\frac{\sqrt{|D_K|}}{\pi}\right)^{1-s} \Gamma^{2} \left(\frac{1-s}{2}\right) \zeta _{K} (1-s) \mbox{ if } D_{K} > 0.$$

\noindent
Now, suppose first that $K$ is an imaginary quadratic field. Then $\zeta_{K} ^{k}(s)$ satisfies the functional equation $$\left(\frac{\sqrt{|D_K|}}{2\pi}\right)^{ks} \Gamma^{k} (s) \zeta _{K}^{k} (s) = \left(\frac{\sqrt{|D_K|}}{2\pi}\right)^{k(1-s)} \Gamma^{k} (1-s) \zeta _{K}^{k} (1-s)$$
Therefore, keeping the same notations as in Theorem \ref{omega-type-theorem-kc-rn}, we have $\Delta (s) \tilde{\phi}(s) = \Delta (1-s) \tilde{\phi} (1-s)$, where $M=k$, $\delta = 1$, $\tilde{\phi} (s)=\left(\frac{\sqrt{|D_K|}}{2\pi}\right)^{ks}  \zeta _{K}^{k} (s)$, $\Delta (s) =\Gamma^{k} (s)$ and $\lambda_n = \left(\frac{\sqrt{|D_K|}}{2\pi}\right)^k \cdot n$.

\vspace{2mm}

\noindent
Thus we see that $\phi (s)=\zeta_{K}^{k}(s)$ is a solution of $\Delta (s) \phi (s) = \Delta (1 - s) \psi (1 - s)$. Hence from Theorem \ref{omega-type-theorem-kc-rn}, our result follows. Similarly, working with the functional equation for real qaudratic fields and using Theorem \ref{omega-type-theorem-kc-rn}, we get the desired result for real quadratic field case.

\bigskip

\noindent
{\bf Acknowledgements.} We are grateful to Prof. R. Thangadurai and Dr. Anirban Mukhopadhyay for their support, encouragement throughout the project and going through the manuscript thoroughly. The second author sincerely acknowledges Prof. R. Balasubramanian and Mr. Priyamvad Srivastav for having several fruitful discussions. The work has been completed when the second author was visiting Harish-Chandra Research Institute and he acknowledges the excellent hospitality and facilities provided by the Institute.

\end{document}